\documentclass [reqno,10pt]{amsart}
\usepackage{amsfonts}
\usepackage{amsmath}
\usepackage{amssymb}
\usepackage{graphicx}
\newtheorem{theorem}{Theorem}[section]
\theoremstyle{plain}

\newtheorem{definition}{Definition}[section]
\newtheorem{example}{Example}[section]

\newtheorem{lemma}{Lemma}[section]

\numberwithin{equation}{section} \textheight 22 true cm \textwidth  15 true cm  
\begin{document}
\title[On semi-slant $\xi^\perp-$Riemannian submersions...]{ON SEMI-SLANT $\xi^\perp-$RIEMANNIAN SUBMERSIONS}
\author{Mehmet Akif Akyol}
\address{Bing\"{o}l University, Faculty of Arts and Sciences, Deparment of
Mathematics, 12000, Bing\"{o}l, Turkey}
\email{mehmetakifakyol@bingol.edu.tr}
\author{Ramazan Sar\i}
\address{Amasya University, Merzifon Vocational Schools, 05300, Amasya,
Turkey}
\email{ramazan.sari@amasya.edu.tr}

\begin{abstract}
The aim of the present paper to define and study semi-slant $\xi^\perp-$Riemannian submersions from Sasakian
manifolds onto Riemannian manifolds as a generalization of anti-invariant $\xi^\perp-$Riemannian submersions,
semi-invariant $\xi^\perp-$Riemannian submersions and slant Riemannian submersions. We
obtain characterizations, investigate the geometry of foliations which arise from the definition of
this new submersion. After we investigate the geometry of foliations, we obtain necessary and sufficient
condition for base manifold to be a locally product manifold and proving new conditions to be totally umbilical
and totally geodesicness, respectively. Moreover, some examples of such submersions are mentioned.
\end{abstract}

\maketitle

\section{Introduction}
A differentiable map $\pi : (M_1, g_1) \longrightarrow (M_2, g_2)$ between Riemannian manifolds $(M_1, g_1)$
and $(M_2, g_2)$ is called a Riemannian submersion if $\pi_*$ is onto and it satisfies
\begin{align}
g_2(\pi_*X_1, \pi_*X_2 )&=g_1(X_1,X_2) \label{1.1}
\end{align}
for $X_1, X_2$ vector fields tangent to $M_1$, where $\pi_*$ denotes the derivative map.
The study of Riemannian submersions were studied by O'Neill \cite{O} and Gray \cite{Gray} see also \cite{FIP}.
Later such submersions according to the conditions on the map $\pi : (M_1, g_1) \longrightarrow (M_2, g_2)$, we have the following submersions:
Riemannian submersions \cite{IMV}, almost Hermitian submersions \cite{W},
invariant submersions (\cite{Chinea, sahin4, sahin5}), anti-invariant submersions (\cite{A1, AF, EM, Gun1, sahin1, sahin4, sahin5}),
lagrangian submersions (\cite{Ta, Ta1}), semi-invariant submersions (\cite{Park1, sahin2}), slant submersions (\cite{EM1, EM2, Gun2, Park4, sahin3}),
semi-slant submersions \cite{A, Gun3, Park2, Park3}, quaternionic submersions \cite{Vilcu, Vilcu1}, hemi-slant submersions (\cite{AY, TSY}),
pointwise slant submersions \cite{LS, SE} etc.
We know that Riemannian submersions have several applications both in mathematics and in physics.
Indeed, Riemannian submersions have their applications in the Kaluza-Klein theory (\cite{BL}, \cite{IV}), supergravity and superstring theories (\cite{IV1}, \cite{M}) and
Yang-Mills theory (\cite{BL1}, \cite{W1}). Recently, in \cite{Lee}, Lee defined anti-invariant $\xi^\perp$-Riemannian
submersions from almost contact metric manifolds and then he studied the
geometry of such maps. Then, in \cite{EM}, Erken and Murathan introduced the
notion of slant submersions from Sasakian manifolds.

On the other hand, as a generalization of anti-invariant $\xi^\perp$-Riemannian submersions,
Akyol et.al in \cite{ASA} defined the notion of semi-invariant $\xi^\perp-$Riemannian submersions
from almost contact metric manifolds and investigate the geometry of such maps.
In this paper, as a generalization of anti-invariant $\xi^\perp-$Riemannian submersions,
semi-invariant $\xi^\perp-$Riemannian submersions and slant Riemannian submersions, we define
semi-slant $\xi^\perp-$Riemannian submersions from Sasakian manifolds onto Riemannian manifolds and investigate
the geometry of the total space and the base space for the existence of such submersions.

The paper is organized as follows. In Sect. 2, we give some basic informations and notions about Riemannian
submersions, the second fundamental form of a map and Sasakian manifolds. In Sect. 3, we define
semi-slant $\xi^\perp-$Riemannian submersions from Sasakian manifolds onto Riemannian manifolds. In Sect. 4, we
investigate the geometry of leaves of the horizontal distribution and the vertical distribution and show that there are
certain product structures on total space of a semi-slant $\xi^\perp-$Riemannian submersion.
In Sect. 5, we find new conditions for a semi-slant $\xi^\perp-$Riemannian submersion to be
totally umbilical and totally geodesicness, respectively. In Sect. 6, we give lots of examples of such submersions.

\section{Riemannian submersions}

Let $({M_1}^{m_1},g_{1})$ and $({M_{2}}^{m_2},g_{2})$ be Riemannian manifolds, where
$\dim(M_1)=m_1,$ $\dim(M_2)=m_2$ and $m_1>m_2.$ A Riemannian submersion $\pi:M_1\longrightarrow
M_2$ is a map of $M_1$ onto $M_2$ satisfying the following axioms:

(i) $\pi$ has maximal rank, and

(ii)The differential $\pi_{\ast}$ preserves the lenghts of horizontal
vectors, that is $\pi_*$ is a linear isometry.

The geometry of Riemannian submersion is characterized by  O'Neill's two $(1,2)$ tensor $\mathcal{T}$ and $\mathcal{A}$ defined as follows:
\begin{equation}
\mathcal{T}(E_1,E_2)=\mathcal{H}\nabla^{^{M_1}}_{\mathcal{V}E_1}\mathcal{V}E_2+\mathcal{V}%
\nabla^{^{M_1}}_{\mathcal{V}E_1}\mathcal{H}E_2  \label{T}
\end{equation}
and
\begin{equation}
\mathcal{A}(E_1,E_2)=\mathcal{H}\nabla^{^{M_1}}_{\mathcal{H}E_1}\mathcal{V}E_2
+\mathcal{V}\nabla^{^{M_1}}_{\mathcal{H}E_1}\mathcal{H}E_2  \label{A}
\end{equation}
for any $E_1, E_2\in\Gamma(M_1),$ where $\nabla^{^{M_1}}$ is the Levi-Civita connection on $g_1.$
Note that we denote the projection morphisms on the vertical distribution and the horizontal
distribution by $\mathcal{V}$ and $\mathcal{H}$, respectively. One can easily see that $\mathcal{T}$ is vertical, $\mathcal{T}_{E_1}=\mathcal{T}_{\mathcal{V}E_1}$ and
$\mathcal{A}$ is horizontal, $\mathcal{A}_{E_1}=\mathcal{A}_{\mathcal{H}E_1}.$ We also note that
$$\mathcal{T}_UV=\mathcal{T}_VU\,\, \textrm{and}\,\,\mathcal{A}_XY=-\mathcal{A}_YX=\frac{1}{2}\mathcal{V}[X,Y], $$
for $X,Y\in \Gamma((ker\pi_*)^{\bot })$ and $U,V\in\Gamma(ker\pi _*).$

On the other hand, from ({\ref{T}}) and ({\ref{A}}), we obtain
\begin{equation}
\nabla^{^M}_{V}W=\mathcal{T}_{V}W+\hat{\nabla}_{V}W;  \label{nvw}
\end{equation}
\begin{equation}
\nabla^{^M}_{V}X=\mathcal{T}_{V}X+\mathcal{H}(\nabla^{^M}_{V}X);  \label{nvx}
\end{equation}
\begin{equation}
\nabla^{^M}_{X}V=\mathcal{V}(\nabla^{^M}_{X}V)+\mathcal{A}_{X}V;  \label{nxv}
\end{equation}
\begin{equation}
\nabla^{^M}_{X}Y=\mathcal{A}_{X}Y+\mathcal{H}(\nabla^{^M}_{X}Y),  \label{nxy}
\end{equation}
for any $X, Y\in\Gamma((ker\pi_{\ast})^{\bot})$ and $V, W\in\Gamma(ker\pi_{\ast}).$ Moreover, if $X$ is basic then $\mathcal{H}(\nabla^{^M}_{V}X)=\mathcal{A}_{X}V.$ It is easy to see that for $U,V\in\Gamma(ker\pi_*),$ $\mathcal{T}_UV$
coincides with the fibers as the second fundamental form and $\mathcal{A}_XY$ reflecting the complete
integrability of the horizontal distribution.\newline
For each $p\in M_2$, $\pi^{-1}(p)$ is an $(m_1-m_2)$ dimensional submanifold of
$M_1$ called a fiber. A vector field on $M_1$ is called \textit{vertical} (resp. hozirontal) if it is always tangent to fibres.
A vector field on $M_1$ is called horizontal if it is always orthogonal to fibres. A
vector field $Z$ on $M_1$ is called basic if $Z$ is horizontal and $\pi-$related
to a vector field $\bar{Z}$ on $M_2$, i.e., $\pi_{\ast}Z_{p}=\bar{Z}_{\pi_{\ast}(p)}$ for
all $p\in M_1$.

\begin{lemma}
(see \cite{FIP}, \cite{O}). Let $\pi : M_1\longrightarrow M_2$ be a Riemannian
submersion. If $X$ and $Y$ basic vector fields on $M_1,$ then we get:

\begin{enumerate}
\item [(i)] $g_1(X,Y)=g_2(\bar{X},\bar{Y})\circ\pi,$

\item [(ii)]$\mathcal{H}[X,Y]$ is a basic and $\pi_*\mathcal{H}[X,Y]= [\bar{X},\bar{Y}]\circ\pi;$

\item [(iii)] $\mathcal{H}(\nabla^{^{M_1}}_{X}Y)$ is a basic, $\pi-$%
related to $(\nabla^{^{M_2}}_{\bar{X}}\bar{Y}),$ where $\nabla^{^{M_1}}$ and
$\nabla^{^{M_2}}$ are the Levi-Civita connection on $M_1$ and $M_2;$

\item [(iv)] $[X,V]\in\Gamma(ker\pi_*)$ is vertical, for any $V\in\Gamma(ker\pi_*).$
\end{enumerate}
\end{lemma}

Let $(M_1,g_{1})$ and $(M_2,g_{2})$ be Riemannian manifolds and $\pi:M_1\longrightarrow M_2$ is a differentiable map.
Then the second fundamental form of $\pi$ is given by
\begin{equation}
(\nabla\pi_*)(X,Y)=\nabla^{^\pi}_{X}\pi_*Y-\pi_*(\nabla_{X}Y)  \label{npixy}
\end{equation}
for $X,Y\in\Gamma(TM_1),$ where $\nabla^{^{\pi}}$ is the pull back connection and $\nabla$ the
Levi-Civita connections of the metrics $g_{1}$ and $g_{2}.$

Finally, let $(M_1,g_1)$ be a $(2m + 1)-$dimensional Riemannian manifold and $TM_1$ denote the tangent bundle of $M_1.$
Then $M_1$ is called an almost contact metric manifold if there exists a tensor $\varphi$ of type $(1,1)$ and global vector
field $\xi$ and $\eta$ is a $1-$form of $\xi$, then we have
\begin{equation}
\varphi^{2}=-I+\eta\otimes\xi,\ \ \ \eta(\xi)=1  \label{s1}
\end{equation}
\begin{equation}
\varphi\xi=0,\ \ \eta o\varphi=0\ \ \ \text{and} \ \ g_{M}(\varphi X,\varphi
Y)=g_{M}(X,Y)-\eta(X)\eta(Y),  \label{s2}
\end{equation}
where $X, Y$ are any vector fields on $M_1.$ In this case, $(\varphi,\xi,\eta,g_1)$ is called the almost contact metric structure of $M_1.$
The almost contact metric manifold $(M_1,\varphi,\xi,\eta,g_1)$ is called a contact metric manifold if
$$\Phi(X,Y)=d\eta(X,Y)$$
for any $X, Y\in\Gamma(TM_1),$ where $\Phi$ is a $2-$form in $M_1$ defined by $\Phi(X,Y)=g_1(X,\varphi Y).$ The $2-$form $\Phi$ is called the
fundamental $2-$form of $M_1.$ A contact metric structure of $M_1$ is said to be normal if
$$[\varphi,\varphi]+2d\eta\otimes\xi=0,$$
where $[\varphi,\varphi]$ is Nijenhuis tensor of $\varphi$. Any normal contact metric manifold is called a Sasakian manifold. Moreover, if $M_1$ is Sasakian \cite{B, SH}, then we have
\begin{equation}
(\nabla^{^{M_1}}_{X}\varphi)Y=g_{M}(X,Y)\xi-\eta(Y)X \ \ \text{and} \ \ \nabla^{^{M_1}}_{X}\xi=-\varphi X,  \label{s3}
\end{equation}
where $\nabla^{^{M_1}}$ is the connection of Levi-Civita covariant differentiation.

\section{Semi-slant $\xi^\perp-$Riemannian submersions}

\begin{definition}\label{def1}
\label{def} Let $(M_1,\varphi,\xi,\eta,g_1)$ be a Sasakian manifold and $%
(M_2,g_{2})$ be a Riemannian manifold. Suppose that there exists a Riemannian
submersion $\pi:M_1 \longrightarrow M_2$ such that $\xi$ is normal to $kerF_*$.
Then $\pi:M_1 \longrightarrow M_2$ is called semi-slant $\xi^\perp-$Riemannian submersion if there
is a distribution $D_{1}\subseteq \ker\pi_{*}$ such that
\begin{equation}
ker\pi_{*}=D_{1}\oplus D_{2},\, \, \, \varphi(D_{1})=D_{1},\label{s4}
\end{equation}
and the angle $\theta=\theta(U)$ between $\varphi U$ and the space $(D_2)_p$ is constant for nonzero
$U\in(D_2)_p$ and $p\in M$, where $D_2$ is the orthogonal complement of $D_1$ in $ker\pi_*$.
As it is, the angle $\theta$ is called the semi-slant angle of the submersion.
\end{definition}

Now, let $\pi$ be a semi-slant $\xi^\perp-$Riemannian submersion from a Sasakian manifold $%
(M_1,\varphi,\xi,\eta,g_{1})$ onto a Riemannian manifold $(M_2,g_{2})$. Then, for $U\in\Gamma(\ker\pi_{*})$, we put
\begin{equation}
U=\mathcal{P}U+\mathcal{Q}U  \label{u}
\end{equation}
where $\mathcal{P}U\in\Gamma(D_1)$ and $\mathcal{Q}U\in\Gamma(D_2).$ For $Z\in\Gamma(TM),$ we have
\begin{equation}
Z=\mathcal{V}Z+\mathcal{H}Z \label{z}
\end{equation}
where $\mathcal{V}Z\in\Gamma(ker\pi_*)$ and $\mathcal{H}Z\in\Gamma((ker\pi_*)^\perp).$ For $V\in\Gamma(ker\pi_*),$ we get
\begin{equation}
\varphi V=\phi V+\omega V  \label{jv}
\end{equation}
where $\phi V$ and $\omega V$ are vertical (resp. horizontal) components of $\varphi V,$ respectively.
Similarly, for any $X\in\Gamma((ker \pi_*)^\perp),$ we have
\begin{equation}
\varphi X=\mathcal{B}X+\mathcal{C}X  \label{jx}
\end{equation}
where $\mathcal{B}X$ (resp. $\mathcal{C}X$) is the vertical part (resp. horizontal part) of $\varphi X.$
Then, the horizontal distribution $(ker\pi_*)^\perp$ is decomposed as
\begin{equation}
(\ker\pi_*)^\perp=\omega D_2\oplus \mu, \label{s5}
\end{equation}
here $\mu$ is the orthogonal complementary distribution of $\omega D_2$ and it is both invariant distribution of $((ker\pi_*)^\perp)$
with respect to $\varphi$ and contains $\xi.$ By (\ref{s2}), (\ref{jv}) and (\ref{jx}), we have
\begin{align}
g_1(\phi U_1, V_1)=-g_1(U_1,\phi V_1) \label{phiu}
\end{align}
and
\begin{align}
g_1(\omega U_1, X)=-g_1(U_1,\mathcal{B}X) \label{omegau}
\end{align}
for $U_1,V_1\in\Gamma(\ker\pi_{*})$ and $X\in\Gamma((\ker\pi_{*})^\perp).$ From (\ref{jv}), (\ref{jx}) and (\ref{s5}), we get:
\begin{lemma}
Let $\pi$ be a semi-slant $\xi^\perp-$Riemannian submersion from a Sasakian manifold $(M_1,\varphi,\xi,\eta,g_1)$ onto a Riemannian manifold
$(M_2,g_{2})$. Then we obtain
\begin{align*}
&(a)\ \ \phi D_1=D_1,\,\,\,\,\,\, (b)\ \ \omega D_1=0,\\
&(c)\ \ \phi D_2\subset D_2,\,\,\,\,\,\, (d)\ \ \mathcal{B}(ker\pi_*)^\perp=D_2,\\
&(e)\ \ \mathcal{T}_{U_1}\xi=\phi U_1,\,\,\,\,(f)\ \ \hat{\nabla}_{U_1}\xi=-\omega U_1,
\end{align*}
for $U_1\in\Gamma(\ker\pi_{*})$ and $\xi\in\Gamma((ker\pi_*)^\perp).$
\end{lemma}
On the other hand, using (\ref{jv}), (\ref{jx}) and the fact that $\varphi^2=-I+\eta\otimes\xi,$ we obtain:
\begin{lemma}
Let $\pi$ be a semi-slant $\xi^\perp-$Riemannian submersion from a Sasakian manifold $(M_1,\varphi,\xi,\eta,g_1)$ onto a Riemannian manifold
$(M_2,g_{2})$. Then we get
\begin{align*}
&(i)\ \ \phi^2+\mathcal{B}\omega=-id,\,\,\,\,\,\,\, (ii)\ \ \mathcal{C}^2+\omega\mathcal{B}=-id,\\
&(iii)\ \ \omega\phi+\mathcal{C}\omega=0,\,\,\,\,\,\,\, (iv)\ \ \mathcal{B}\mathcal{C}+\phi\mathcal{B}=0,
\end{align*}
where $I$ is the identity operator on the space of $\pi.$
\end{lemma}

Let $(M_1,\varphi,\xi,\eta,g_1)$ be a Sasakian manifold and $(M_2,g_2)$ be a Riemannian manifold.
Let $\pi:(M_1,\varphi,\xi,\eta,g_1)\longrightarrow (M_2,g_2)$
be a semi-slant $\xi^\perp-$Riemannian submersion. We now examine how the Sasakian
structure on $M_1$ effects the tensor fields $\mathcal{T}$ and $\mathcal{A}$ of a semi-slant $\xi^\perp-$Riemannian
submersion $\pi:(M_1,\varphi,\xi,\eta,g_1)\longrightarrow (M_2,g_2)$.

\begin{lemma}
Let $(M_1,\varphi,\xi,\eta,g_1)$ be a Sasakian manifold and $(M_2,g_2)$ a Riemannian manifold.
Let $\pi:(M_1,\varphi,\xi,\eta,g_1)\longrightarrow (M_2,g_2)$ be a semi-slant $\xi^\perp-$Riemannian submersion. Then we have
\begin{align}
\mathcal{B}\mathcal{T}_{U}V+\phi\hat{\nabla}_{U}V&=\hat{\nabla}_{U}\phi V+\mathcal{T}_{U}\omega V,  \label{s6}\\
g_1(U,V)\xi+\mathcal{C}\mathcal{T}_{U}V+\omega\hat{\nabla}_{U}V&=\mathcal{T}_{U}\phi V+\mathcal{H}\nabla^{^{M_1}}_{U}\omega V,  \label{s7}
\end{align}
\begin{align}
\phi \mathcal{T}_{U}X+\mathcal{B}\nabla^{^{M_1}}_{U}X-\eta(X)U&=\hat{\nabla}_{U}\mathcal{B}X+\mathcal{T}_{U}\mathcal{C}X, \label{s8}\\
\omega \mathcal{T}_{U}X+\mathcal{C}\nabla^{^{M_1}}_{U}X&=\mathcal{T}_{U}\mathcal{B}X+\mathcal{H}\nabla^{^{M_1}}_{U}\mathcal{C}X, \label{s9}
\end{align}
\begin{align}
g_1(X,Y)\xi-\omega \mathcal{A}_{X}Y+\mathcal{C}\mathcal{H}\nabla^{^{M_1}}_{X}Y&=\mathcal{A}_{X}\mathcal{B}Y+\nabla^{^{M_1}}_{X}\mathcal{C}Y+\eta(Y)X, \label{s10}\\
\phi \mathcal{A}_{X}Y+\mathcal{B}\mathcal{H}\nabla^{^{M_1}}_{X}Y&=\mathcal{V}\nabla^{^{M_1}}_{X}\mathcal{B}Y+\mathcal{A}_{X}\mathcal{C}Y,  \label{s11}
\end{align}
for all $X,Y\in\Gamma((\ker\pi_{*})^\perp)$ and $U,V\in\Gamma(\ker\pi_{*})$.
\end{lemma}

\begin{proof}
Given $U,V\in\Gamma(ker\pi_*)$, by virtue of (\ref{s3}) and (\ref{jv}), we have
$$g_1(U,V)\xi-\eta(V)U=\nabla^{^{M_1}}_{U}\phi V+\nabla^{^{M_1}}_{U}\omega V-\varphi \nabla^{^{M_1}}_{U}V.$$
By using (\ref{nvw}), (\ref{nvx}), (\ref{jv}) and (\ref{jx}), we get
\begin{align}
g_1(U,V)\xi&=\mathcal{T}_{U}\phi V+\hat{\nabla}_{U}\phi V+\mathcal{T}_{U}\omega V+\mathcal{H}\nabla^{^{M_1}}_{U}\omega V \notag\\
&-\mathcal{B}\mathcal{T}_{U}V-\mathcal{C}\mathcal{T}_{U}V-\phi \hat{\nabla}_{U}V-\omega \hat{\nabla}_{U}V. \label{s12}
\end{align}
In (\ref{s12}), comparing horizontal and vertical parts, we get (\ref{s6}) and (\ref{s7}). The other assertions can
be obtained in a similar method.

\end{proof}

As the proof of the following theorem is similar to semi-slant submanifolds (Theorem 5.1 of \cite{CCFF}), we omit it.

\begin{theorem}\label{teo1}
Let $\pi:(M_1,\varphi,\xi,\eta,g_1)\longrightarrow (M_2,g_2)$ be a semi-slant $\xi^\perp-$Riemannian submersion
from a Sasakian manifold $(M_1,\varphi,\xi,\eta,g_1)$ onto a Riemannian manifold $(M_2,g_2).$ Then we have
\begin{equation}
\phi^{2}W=-\cos^2{\theta}W,\,\,\ W\in\Gamma(D_2),\label{phiv}
\end{equation}
where $\theta$ denotes the semi-slant angle of $D_2$.
\end{theorem}

By using (\ref{s2}), (\ref{jv}), (\ref{phiu}), (\ref{omegau}) we get:

\begin{lemma}
Let $\pi:(M_1,\varphi,\xi,\eta,g_1)\longrightarrow (M_2,g_2)$ be a semi-slant $\xi^\perp-$Riemannian submersion
from a Sasakian manifold $(M_1,\varphi,\xi,\eta,g_1)$ onto a Riemannian manifold $(M_2,g_2)$ with a semi-slant angle $\theta.$
Then we have
\begin{align}
g_1(\phi W_1,\phi W_2)&=\cos^{2}\theta g_1(W_1,W_2),\label{s15}\\
g_1(\omega W_1,\omega W_2)&=\sin^{2}\theta g_1(W_1,W_2), \label{s16}
\end{align}
for any $W_1, W_2\in\Gamma(D_2).$
\end{lemma}

\section{Integrability, Totally Geodesicness and Decomposition Theorems}

In this section, we shall study the integrability and totally geodesicness of the distributions
which are involved in the definition of a semi-slant $\xi^\perp-$Riemannian submersions and obtain decomposition theorems
of such submersions.

\begin{theorem}
Let $\pi$ be a semi slant $\xi^\perp-$Riemannian submersion from a Sasakian manifold $(M_1,\varphi,\xi,\eta,g_1)$
onto a Riemannian manifold $(M_2,g_2)$ with a semi-slant angle $\theta.$ Then
\begin{enumerate}
\item [(i)] $D_{1}$ is integrable $\Leftrightarrow$
$\begin{aligned}[t]
(\nabla \pi _{\ast })(U,\varphi V)-(\nabla \pi _{\ast })(V,\varphi U)\notin\Gamma (\pi _{\ast}\mu)
\end{aligned}$
\item [(ii)] $D_{2}$ is integrable $\Leftrightarrow$
$\begin{aligned}[t]
g_{2}(\pi _{\ast }\omega W,(\nabla \pi _{\ast })(Z,\varphi U))+g_{2}(\pi _{\ast
}\omega Z,(\nabla \pi _{\ast })(W,\varphi U))&=g_{1}(\phi W,\hat{\nabla}_{Z}\varphi U)\\
&+g_{1}(\phi Z,\hat{\nabla}_{W}\varphi U)
\end{aligned}$
\end{enumerate}
for $U,V\in \Gamma (D_{1})$ and $Z,W\in \Gamma (D_{2}).$
\end{theorem}

\begin{proof}
Given $U, V\in\Gamma(D_{1})$ and $X\in\Gamma((ker\pi_*)^\perp)$, since $[U,V]\in\Gamma(ker\pi_*)$, we have $g_1([U,V],X)=0.$ Thus $D_1$ is integrable
$\Leftrightarrow$ $g_1([U,V],Z)=0$ for $Z\in\Gamma(D_2).$ Since $M_1$ is a Sasakian manifold, by (\ref{s2}) and (\ref{s3}), we have
\begin{align}
g_{1}(\nabla^{^{M_1}}_{U}V,Z)&=g_{1}(\nabla^{^{M_1}}_{U}\varphi V-g_{1}(U,V)\xi -\eta(V)U,\varphi Z)\notag\\
&=g_{1}(\nabla^{^{M_1}}_{U}\varphi V,\varphi Z).\label{4s1}
\end{align}%
Using (\ref{jv}) in (\ref{4s1}), we get
\begin{align*}
g_{1}([U,V],Z)&=-g_{1}(\nabla^{^{M_1}}_{U}V,\varphi \phi Z)+g_{1}(\mathcal{H}\nabla^{^{M_1}}_{U}\varphi
V,wZ)-g_{1}(\nabla^{^{M_1}}_{V}U,\varphi \phi Z)-g_{1}(\mathcal{H}\nabla^{^{M_1}}_{V}\varphi U,wZ).
\end{align*}
Then, by (\ref{phiv}) and (\ref{npixy}), we conclude that
\begin{align*}
g_{1}([U,V],Z)&=\cos ^{2}\theta g_{1}(\nabla^{^{M_1}}_{U}V,Z)-g_{2}((\nabla \pi _{\ast
})(U,\varphi V)+\nabla _{U}^{\pi }\pi _{\ast }\varphi V,\pi _{\ast }wZ) \\
&-\cos ^{2}\theta g_{1}(\nabla^{^{M_1}}_{V}U,Z)+g_{2}((\nabla \pi _{\ast
})(V,\varphi U)+\nabla _{V}^{\pi }\pi _{\ast }\varphi U,\pi _{\ast }wZ).
\end{align*}
After some calculations, we obtain
\begin{align*}
(\sin^{2}\theta)g_{1}([U,V],Z)&=-g_{2}((\nabla \pi _{\ast })(U,\varphi
V)-(\nabla \pi _{\ast })(V,\varphi U),\pi _{\ast }wZ),
\end{align*}
which proves (i). The other assertion can be obtained in a similar method.
\end{proof}

We now investigate the geometry of leaves of $D_1$ and $D_2.$

\begin{theorem}\label{d1par}
Let $\pi$ be a semi slant $\xi^\perp-$Riemannian submersion from a Sasakian manifold $(M_1,\varphi,\xi,\eta,g_1)$
onto a Riemannian manifold $(M_2,g_2)$ with a semi-slant angle $\theta.$ Then the distribution $D_{1}$ is parallel if and only if%
\begin{align}
g_{2}((\nabla \pi_{\ast })(U,\varphi V),\pi _{\ast }\omega Z)&=g_{1}(\mathcal{T}_{U}\omega\phi Z,V) \label{4s2}
\end{align}
and%
\begin{align}
-g_{2}((\nabla \pi_{\ast })(U,\varphi V),\pi _{\ast }\mathcal{C}X)&=g_{1}(V,\hat{\nabla}_{U}\phi\mathcal{B}X+\mathcal{T}_{U}\omega\mathcal{B}X)
+g_{1}(V,\varphi U)\eta(X)\label{4s3}
\end{align}
for $U, V\in \Gamma (D_{1}), Z\in \Gamma (D_{2})$ and $X\in\Gamma((\ker \pi _{\ast})^{\bot })$.
\end{theorem}

\begin{proof}
By virtue of (\ref{4s1}), (\ref{jv}) and (\ref{nvw}), we have
\begin{align*}
g_{1}(\nabla^{^{M_1}}_{U}V,Z)&=-g_{1}(\nabla^{^{M_1}}_{U}V,\phi ^{2}Z)-g_{1}(\nabla^{^{M_1}}_{U}V,\omega \phi Z)
+g_{1}(\mathcal{H}\nabla^{^{M_1}}_{U}\varphi V,\omega Z)
\end{align*}
for $U, V\in \Gamma (D_{1})$ and $Z\in \Gamma (D_{2}).$ Since $\pi$ is a semi-slant $\xi^\perp-$Riemannian submersion, using (\ref{nvw}) and (\ref{npixy}), we arrive
\begin{align*}
g_{1}(\nabla^{^{M_1}}_{U}V,Z)&=\cos ^{2}\theta g_{1}(\nabla^{^{M_1}}_{U}V,Z)-g_{1}(\mathcal{T}_{U}V,w\phi Z)+g_{2}((\nabla
\pi _{\ast })(U,\varphi V),\pi _{\ast }(wZ))
\end{align*}%
or
\begin{align*}
\sin ^{2}\theta g_{1}(\nabla^{^{M_1}}_{U}V,Z)=-g_{1}(\mathcal{T}_{U}w\phi Z,V)+g_{2}((\nabla
\pi _{\ast })(U,\varphi V),\pi _{\ast }(wZ)),
\end{align*}
which gives (\ref{4s2}). On the other hand, from (\ref{s2}) and (\ref{s3}), we get
\begin{align*}
g_{1}(\nabla^{^{M_1}}_{U}V,X)&=g_{1}(\nabla^{^{M_1}}_{U}\varphi V,\varphi X)+g_{1}(V,\varphi U)\eta (X).
\end{align*}
for $U,V\in \Gamma (D_{1})$ and $X\in \Gamma ((\ker \pi_{\ast })^{\bot }).$ Then using (\ref{jv}) and (\ref{jx}), we conlude that
\begin{align*}
g_{1}(\nabla^{^{M_1}}_{U}V,X)&=g_{1}(V,\nabla^{^{M_1}}_{U}\phi \mathcal{B}X)+g_{1}(V,\nabla^{^{M_1}}_{U}\omega \mathcal{B}X)
+g_{1}(\mathcal{C}X,\mathcal{H}\nabla^{^{M_1}}_{U}\varphi V)+g_{1}(V,\varphi U)\eta (X).
\end{align*}
Also, using (\ref{nvw}), (\ref{nvx}) and the character of $\pi$, we get
\begin{align*}
g_{1}(\nabla^{^{M_1}}_{U}V,X)&=g_{1}(V,\mathcal{T}_{U}\phi \mathcal{B}X+\hat{\nabla}_{U}\phi
\mathcal{B}X)+g_{1}(V,\mathcal{T}_{U}\omega\mathcal{B}X+\mathcal{H}\nabla^{^{M_1}}_{U}\omega\mathcal{B}X) \\
&-g_{2}(\pi_{\ast}(\mathcal{C}X),\pi_{\ast}(\mathcal{H}\nabla^{^{M_1}}_{U}\varphi V))+g_{1}(V,\varphi U)\eta(X).
\end{align*}
Then (\ref{npixy}) imply
\begin{align*}
g_{1}(\nabla^{^{M_1}}_{U}V,X)&=g_{1}(V,\hat{\nabla}_{U}\phi
\mathcal{B}X)+g_{1}(V,\mathcal{T}_{U}\omega\mathcal{B}X)+g_{2}((\nabla\pi _{\ast})(U,\varphi V),\pi_{\ast}\mathcal{C}X)\\
&+g_{1}(V,\varphi U)\eta(X),
\end{align*}
which gives (\ref{4s3}).
\end{proof}

For $D_{2}$, we get:

\begin{theorem}\label{d2par}
Let $\pi$ be a semi slant $\xi^\perp-$Riemannian submersion from a Sasakian manifold $(M_1,\varphi,\xi,\eta,g_1)$
onto a Riemannian manifold $(M_2,g_2)$ with a semi-slant angle $\theta.$ Then the distribution $D_{2}$ is parallel if and only if%
\begin{align}
g_{2}(\pi _{\ast }\omega W,(\nabla \pi_{\ast})(Z,\varphi U))&=g_{1}(\phi W,\hat{\nabla}_{Z}\varphi U) \label{4s4}
\end{align}
and
\begin{align}
g_{2}((\nabla \pi _{\ast })(Z,\omega W),\pi _{\ast}(X))-g_{2}((\nabla\pi _{\ast
})(Z,\omega \phi W),\pi _{\ast}(X))&=g_{1}(\mathcal{T}_{Z}\omega W,\mathcal{B}X)+g_{1}(W,\varphi Z)\eta (X) \label{4s5}
\end{align}
for any $Z,W\in \Gamma (D_{2}), U\in\Gamma(D_1)$ and $X\in\Gamma((ker\pi_*)^\perp).$
\end{theorem}

\begin{proof}
Given $Z,W\in \Gamma (D_{2})$ and $U\in\Gamma(D_1),$ using (\ref{4s1}), (\ref{jv}) and (\ref{nvw}), we get
\begin{align*}
g_{1}(\nabla^{^{M_1}}_{Z}W,U)&=-g_{1}(\phi W,\hat{\nabla}_{Z}\varphi U)-g_{2}(\omega W,\pi _{\ast}(\mathcal{H}\nabla^{^{M_1}}_{Z}\varphi U)).
\end{align*}
Since $\pi$ is a semi-slant $\xi^\perp-$Riemannian submersion, using (\ref{nvw}) and (\ref{npixy}), we arrive
\begin{align*}
g_{1}(\nabla^{^{M_1}}_{Z}W,U)&=-g_{1}(\phi W,\hat{\nabla}_{Z}\varphi U)+g_{2}(\pi_{\ast}\omega W,(\nabla \pi _{\ast })(Z,\varphi U))
\end{align*}%
which completes (\ref{4s4}). On the other hand, by using (\ref{4s1}) and(\ref{jv}) we get
\begin{align*}
g_{1}(\nabla^{^{M_1}}_{Z}W,X)&=-g_{1}(\varphi\nabla^{^{M_1}}_{Z}\varphi W,X)+g_{1}(W,\varphi Z)\eta (X)
\end{align*}
for all $Z,W\in \Gamma (D_{2}), U\in\Gamma(D_1)$ and $X\in\Gamma((ker\pi_*)^\perp).$ Then (\ref{nvx}) and (\ref{phiv}) imply that
\begin{align*}
g_{1}(\nabla^{^{M_1}}_{Z}W,X)&=\cos ^{2}\theta g_{1}(\nabla^{^{M_1}}_{Z}W,X)-g_{1}(\mathcal{H}\nabla^{^{M_1}}_{Z}\omega\phi W,X)\\
&+g_{1}(\nabla^{^{M_1}}_{Z}\omega W,\mathcal{B}X)+g_{1}(\nabla^{^{M_1}}_{Z}\omega W,\mathcal{C}X)+g_{1}(W,\varphi Z)\eta(X).
\end{align*}
Now, using the character of $\pi$ and (\ref{nvx}), we get
\begin{align*}
\sin^{2}\theta g_{1}(\nabla^{^{M_1}}_{Z}W,X)&=-g_{2}(\pi_{\ast}(\mathcal{H}\nabla\nabla^{^{M_1}}_{Z}\omega\phi W),\pi _{\ast}(X))
+g_{2}(\mathcal{T}_{Z}\omega W,\mathcal{B}X)\\
&+g_{2}(\pi_*(\nabla^{^{M_1}}_{Z}\omega\phi W),\pi_*(X))+g_{1}(W,\varphi Z)\eta(X).
\end{align*}
After some calculations, we have
\begin{align*}
\sin^{2}\theta g_{1}(\nabla^{^{M_1}}_{Z}W,X)&=g_{2}((\nabla\pi_{\ast})(Z,\omega\phi W),\pi_{\ast }(X))+g_{1}(\mathcal{T}_{Z}\omega W,\mathcal{B}X) \\
&-g_{2}((\nabla\pi_{\ast})(Z,\omega W),\pi_{\ast}(X))+g_{1}(W,\varphi Z)\eta(X),
\end{align*}
which gives (\ref{4s5}).
\end{proof}

Since $(ker\pi_*)$ is integrable, we only study the integrability of the distribution $(ker\pi_{*})^\perp$ and
then we investigate the geometry of leaves of $ker\pi_{*}$ and $(ker\pi_{*})^\perp$.

\begin{theorem}
Let $\pi$ be a semi slant $\xi^\perp-$Riemannian submersion from a Sasakian manifold $(M_1,\varphi,\xi,\eta,g_1)$
onto a Riemannian manifold $(M_2,g_2)$ with a semi-slant angle $\theta.$ Then the distribution $(\ker \pi _{\ast })^{\bot }$ is integrable if and only if%
\begin{align}
g_{2}((\nabla \pi_{\ast})(Y,\phi V),\pi_{\ast}(X))+g_{2}((\nabla \pi
_{\ast})(X,\phi V),\pi_{\ast}(X))&=g_{1}(\phi V,\mathcal{V}(\nabla^{^{M_1}}_{X}\mathcal{B}Y+%
\nabla^{^{M_1}}_{Y}\mathcal{B}X)) \label{4s6}
\end{align}
and%
\begin{align}
g_{2}((\nabla\pi _{\ast})(X,\mathcal{C}Y)-(\nabla\pi_{\ast})(Y,\mathcal{C}X),\pi_{\ast
}\omega W)&=g_{1}(\mathcal{A}_{X}\mathcal{B}Y+\mathcal{A}_{Y}\mathcal{B}X,\omega W)\notag\\
&+\eta(Y)g_1(X,\omega W)-\eta(X)g_1(Y,\omega W) \label{4s7}
\end{align}
for $X, Y\in\Gamma((ker\pi_{\ast })^{\bot}), V\in\Gamma(D_{1})$ and $W\in\Gamma(D_{2}).$
\end{theorem}
\begin{proof}
From (\ref{4s1}), (\ref{s2}) and (\ref{s3}), we have
\begin{align*}
g_{1}([X,Y],V)&=g_{1}(\nabla^{^{M_1}}_{X}\varphi Y,\varphi V)-g_{1}(\nabla^{^{M_1}}_{Y}\varphi X,\varphi V)
\end{align*}
for $X,Y\in \Gamma((\ker \pi _{\ast })^{\bot })$ and $V\in \Gamma (D_{1}).$ Then by (\ref{jx}), we have
\begin{align*}
g_{1}([X,Y],V)&=-g_{1}(\mathcal{B}Y,\nabla^{^{M_1}}_{X}\varphi V)-g_{1}(\mathcal{C}Y,\nabla^{^{M_1}}_{X}\varphi
V)+g_{1}(\mathcal{B}X,\nabla^{^{M_1}}_{Y}\varphi V)+g_{1}(\mathcal{C}X,\nabla^{^{M_1}}_{Y}\varphi V).
\end{align*}
From (\ref{nvx}), we get
\begin{align*}
g_{1}([X,Y],V)&=g_{1}(\varphi V,\mathcal{A}_{Y}\mathcal{B}X+\mathcal{V}\nabla^{^{M_1}}_{X}\mathcal{B}Y)-g_{2}(\pi _{\ast
}(\mathcal{C}Y),\pi_{\ast}(\nabla^{^{M_1}}_{X}\varphi V))\\
&-g_{1}(\varphi V,\mathcal{A}_{X}\mathcal{B}Y+\mathcal{V}\nabla^{^{M_1}}_{Y}\mathcal{B}X)
-g_{2}(\pi_{\ast}(\mathcal{C}X),\pi_{\ast}(\nabla^{^{M_1}}_{Y}\varphi V)).
\end{align*}
Using (\ref{npixy}), we obtain
\begin{align*}
g_{1}([X,Y],V)&=g_{1}(\varphi V,\mathcal{V}(\nabla^{^{M_1}}_{X}\mathcal{B}Y-\nabla^{^{M_1}}_{Y}\mathcal{B}X))+g_{2}(\pi
_{\ast }(\mathcal{C}Y),(\nabla \pi _{\ast})(X, \varphi V))\\
&-g_{2}(\pi _{\ast}(\mathcal{C}X),(\nabla\pi _{\ast})(Y,\varphi V)),
\end{align*}%
which gives (\ref{4s6}). In a similar way, by virtue of (\ref{4s1}), (\ref{s2}) and (\ref{s3}), we have
\begin{align*}
g_{1}([X,Y],W)&=g_{1}(\varphi\nabla^{^{M_1}}_{X}Y,\phi W)+g_{1}(\varphi\nabla^{^{M_1}}_{X}Y,\omega W)+\eta
(Y)g_{1}(X,\omega W) \\
&-g_{1}(\varphi\nabla^{^{M_1}}_{Y}X,\phi W)-g_{1}(\varphi\nabla^{^{M_1}}_{Y}X,\omega W)-\eta
(X)g_{1}(Y,\omega W)
\end{align*}
for $X,Y\in \Gamma ((\ker \pi _{\ast })^{\bot })$ and $W\in \Gamma (D_{2}).$ Then by (\ref{phiv}) and (\ref{nxy}), we arrive
\begin{align*}
g_{1}([X,Y],W)&=-g_{1}(\nabla^{^{M_1}}_{X}Y,\phi ^{2}W)-g_{1}(\nabla^{^{M_1}}_{X}Y,\omega\phi W)+g_{1}(\nabla^{^{M_1}}_{X}\mathcal{B}Y,\omega W)
+g_{1}(\nabla^{^{M_1}}_{X}\mathcal{C}Y,\omega W) \\
&-g_{1}(\nabla^{^{M_1}}_{Y}X,\phi ^{2}W)-g_{1}(\nabla^{^{M_1}}_{Y}X,\omega\phi W)+g_{1}(\nabla^{^{M_1}}_{Y}\mathcal{B}X,\omega W)
+g_{1}(\nabla^{^{M_1}}_{Y}\mathcal{C}X,\omega W)\\
&+\eta (Y)g_{1}(X,\omega W)-\eta (X)g_{1}(Y,\omega W).
\end{align*}
Now, using the character of $\pi$, (\ref{nxy}) and (\ref{npixy}) imply that
\begin{align*}
g_{1}([X,Y],W)&=\cos ^{2}\theta g_{1}([X,Y],W)+g_{2}((\nabla \pi _{\ast })(X,Y),\omega\phi
W)+g_{1}(\mathcal{A}_{X}\mathcal{B}Y,\omega W)\\
&-g_{2}((\nabla \pi _{\ast })(X,\mathcal{C}Y),\pi _{\ast}\omega W)-g_{2}((\nabla \pi _{\ast })(Y,X),\omega\phi W)+g_{1}(\mathcal{A}_{Y}\mathcal{B}X,\omega W)\\
&+g_{2}((\nabla\pi _{\ast })(Y,\mathcal{C}X),\pi _{\ast}\omega W)+\eta (Y)g_{1}(X,\omega W)-\eta (X)g_{1}(Y,\omega W)
\end{align*}
so with some elementary calculations, we find
\begin{align*}
\sin ^{2}\theta g_{1}([X,Y],W)&=g_{2}((\nabla \pi _{\ast })(Y,\mathcal{C}X)-(\nabla
\pi _{\ast })(X,\mathcal{C}Y),\pi _{\ast }\omega W)+g_{1}(\mathcal{A}_{X}\mathcal{B}Y+\mathcal{A}_{Y}\mathcal{B}X,\omega W) \\
&+\eta (Y)g_{1}(X,\omega W)-\eta (X)g_{1}(Y,\omega W),
\end{align*}
which completes (\ref{4s7}).
\end{proof}

For the geometry of leaves $(\ker \pi _{\ast })^{\bot }$, we obtain:
\begin{theorem}\label{yataypar}
Let $\pi$ be a semi slant $\xi^\perp-$Riemannian submersion from a Sasakian manifold $(M_1,\varphi,\xi,\eta,g_1)$
onto a Riemannian manifold $(M_2,g_2)$ with a semi-slant angle $\theta.$ Then the distribution $(\ker \pi _{\ast })^{\bot }$ is parallel if and only if%
\begin{align}
g_{1}(V,\mathcal{V}\nabla^{^{M_1}}_{X}\phi \mathcal{B}Y+\mathcal{A}_{X}\omega\mathcal{B}Y)&=g_{2}(\pi_*(\mathcal{C}Y),(\nabla \pi_*)(X,\varphi V)) \label{4s8}
\end{align}
and%
\begin{align}
g_{1}(\mathcal{A}_{X}\omega W,\mathcal{B}Y)+\eta(Y)g_{1}(X,\omega W)&=g_{2}((\nabla \pi_*)(X,Y),\pi_*\omega\phi W)\notag\\
&-g_{2}((\nabla \pi_*)(X,\mathcal{C}Y),\pi_*\omega W),\label{4s9}
\end{align}
for $X, Y\in \Gamma((\ker \pi_*)^{\bot}), V\in\Gamma(D_1)$ and $W\in\Gamma(D_{2}).$
\end{theorem}

\begin{proof}
Given $X, Y\in \Gamma ((\ker \pi _{\ast })^{\bot })$ and $V\in \Gamma (D_{1})$, by (\ref{s2}) and (\ref{s3}), we have
\begin{align*}
g_{1}(\nabla^{^{M_1}}_{X}Y,V)&=g_{1}(\nabla^{^{M_1}}_{X}\varphi Y,\varphi V).
\end{align*}
Thus, from (\ref{jx}), we find
\begin{align*}
g_{1}(\nabla^{^{M_1}}_{X}Y,V)&=-g_{1}(\mathcal{B}Y,\nabla^{^{M_1}}_{X}\varphi V)-g_{1}(\mathcal{C}Y,\nabla^{^{M_1}}_{X}\varphi V).
\end{align*}
Taking into account that $\pi$ is a semi-slant $\xi^\perp-$Riemannian submersion, we get
\begin{align*}
g_{1}(\nabla^{^{M_1}}_{X}Y,V)&=-g_{1}(V,\mathcal{A}_{X}\phi \mathcal{B}Y+\mathcal{V}\nabla^{^{M_1}}_{X}\phi \mathcal{B}Y)-g_{1}(V,\mathcal{A}_{X}\omega\mathcal{B}Y+%
\mathcal{H}\nabla^{^{M_1}}_{X}\omega\mathcal{B}Y)\\
&+g_{2}(\pi_*(CY),\pi_*(\nabla^{^{M_1}}_{X}\varphi V)).
\end{align*}
By (\ref{npixy}), we obtain
\begin{align*}
g_{1}(\nabla^{^{M_1}}_{X}Y,V)&=-g_{1}(V,\mathcal{V}\nabla^{^{M_1}}_{X}\phi \mathcal{B}Y+\mathcal{A}_{X}\omega \mathcal{B}Y)+g_{2}(\pi _*(CY),(\nabla \pi _*)(X,\varphi V)),
\end{align*}
which gives (\ref{4s8}). On the other hand, for $X,Y\in \Gamma((\ker \pi _*)^\perp),$ $W\in \Gamma(D_{2})$, by virtue of (\ref{s2}), (\ref{s3}), we obtain
\begin{align*}
g_{1}(\nabla^{^{M_1}}_{X}Y,W)&=g_{1}(\nabla^{^{M_1}}_{X}\varphi Y,\phi W)+g_{1}(\nabla^{^{M_1}}_{X}\varphi Y,\omega W)-\eta(Y)g_{1}(X,\omega W).
\end{align*}
From (\ref{jv}) and (\ref{jx}), we arrive
\begin{align*}
g_{1}(\nabla^{^{M_1}}_{X}Y,W)&=-g_{1}(\nabla^{^{M_1}}_{X}Y,\phi ^{2}W)-g_{1}(\nabla^{^{M_1}}_{X}Y,\omega\phi W)
+g_{1}(\nabla^{^{M_1}}_{X}\mathcal{B}Y,\omega W) \\
&+g_{1}(\nabla^{^{M_1}}_{X}\mathcal{C}Y,\omega W)-\eta (Y)g_{1}(X,\omega W).
\end{align*}
Since $\pi$ is a semi-slant $\xi^\perp-$Riemannian submersion, using (\ref{phiv}) and the character of $\pi$ we get
\begin{align*}
g_{1}(\nabla^{^{M_1}}_{X}Y,W)&=\cos ^{2}\theta g_{1}(\nabla^{^{M_1}}_{X}Y,W)-g_{1}(\pi_*(\mathcal{H}\nabla^{^{M_1}}_{X}Y),\pi_*(\omega\phi W))+g_{1}(\mathcal{A}_{X}\mathcal{B}Y,\omega W) \\
&+g_{1}(\pi_*(\mathcal{H}\nabla^{^{M_1}}_{X}\mathcal{C}Y),\pi_*(\omega W))-\eta (Y)g_{1}(X,\omega W)
\end{align*}
Now, using (\ref{npixy}), we obtain
\begin{align*}
\sin ^{2}\theta g_{1}(\nabla^{^{M_1}}_{X}Y,W) &=g_{2}((\nabla \pi _{\ast})(X,Y),\pi _{\ast }\omega \phi W)+g_{1}(\mathcal{A}_{X}\mathcal{B}Y,\omega W)
-g_{2}((\nabla \pi _{\ast })(X,\mathcal{C}Y),\pi _{\ast }\omega W)\\
&-\eta (Y)g_{1}(X,\omega W)
\end{align*}%
which gives (\ref{4s9}).
\end{proof}

Similarly, we get:

\begin{theorem}\label{dikeypar}
Let $\pi$ be a semi slant $\xi^\perp-$Riemannian submersion from a Sasakian manifold $(M_1,\varphi,\xi,\eta,g_1)$
onto a Riemannian manifold $(M_2,g_2)$ with a semi-slant angle $\theta.$ Then the distribution $(ker\pi_*)$ is parallel if and only if%
\begin{align}
g_{1}(\omega V,\mathcal{T}_{U}\mathcal{B}X)+g_{1}(V,\phi U)\eta (X)&=g_{2}((\nabla \pi _{\ast
})(U,\mathcal{C}X),\pi _{\ast }\omega V)-g_{2}((\nabla \pi _{\ast })(U,X),\pi _{\ast }\omega\phi V) \label{4s10}
\end{align}
\end{theorem}
for any $U\in\Gamma(D_1), V\in \Gamma (D_2)$ and $X\in \Gamma ((ker \pi _{\ast })^{\bot }).$
\begin{proof}
By virtue of (\ref{s2}) and (\ref{s3}), we have
\begin{align*}
g_{1}(\nabla^{^{M_1}}_{U}V,X) &=g_{1}(\nabla^{^{M_1}}_{U}\varphi V,\varphi X)
\end{align*}
for $U\in\Gamma(D_1), V\in \Gamma (D_2)$ and $X\in \Gamma ((ker \pi _{\ast })^{\bot }).$ Then, from (\ref{jv}), we arrive
\begin{align*}
g_{1}(\nabla^{^{M_1}}_{U}V,X) &=-g_{1}(\phi V,\varphi\nabla^{^{M_1}}_{U}X)-g_{1}(\omega V,\nabla^{^{M_1}}_{U}\varphi X).
\end{align*}
Now, by (\ref{jv}), (\ref{jx}) and (\ref{phiv}), we obtain
\begin{align*}
g_{1}(\nabla^{^{M_1}}_{U}V,X) &=g_{1}(\phi ^{2}V,\nabla^{^{M_1}}_{U}X)+g_{1}(\omega\phi V,\nabla^{^{M_1}}_{U}X)
-g_{1}(\omega V,\mathcal{T}_{U}\mathcal{B}X)\\
&-g_{1}(\omega V,\mathcal{H}\nabla^{^{M_1}}_{U}\mathcal{C}X)-g_{1}(V,\phi U)\eta (X).
\end{align*}
By virtue of (\ref{npixy}), (\ref{nvw}) and the character of $\pi,$ we arrive
\begin{align*}
g_{1}(\nabla^{^{M_1}}_{U}V,X) &=-\cos ^{2}\theta g_{1}(V,\nabla^{^{M_1}}_{U}X)
+g_{1}(\pi_*(\omega\phi V),\pi_*(\mathcal{H}\nabla^{^{M_1}}_{U}X))
-g_{1}(\omega V,\mathcal{T}_{U}\mathcal{B}X) \\
&-g_{1}(\pi_*(\omega V),\pi_*(\mathcal{H}\nabla^{^{M_1}}_{U}\mathcal{C}X))-g_{1}(V,\phi U)\eta (X)
\end{align*}%
so with some elementary calculations, we get
\begin{align*}
\sin ^{2}\theta g_{1}(\nabla^{^{M_1}}_{U}V,X)&=-g_{2}((\nabla \pi_*)(U,X),\pi
_*\omega\phi V)-g_{1}(\omega V,\mathcal{T}_{U}\mathcal{B}X)\\
&+g_{2}((\nabla \pi_*)(U,\mathcal{C}X),\pi_*\omega V)-g_{1}(V,\phi U)\eta (X)
\end{align*}
which completes (\ref{4s9}).
\end{proof}

We now recall the following characterization for locally (usual) product Riemannian manifold from  \cite{PR}.
Let $g$ be a Riemannian metric tensor on the manifold $M = M_1\times M_2$ and assume
that the canonical foliations $D_{M_1}$ and $D_{M_2}$ intersect perpendicularly everywhere.
Then $g$ is the metric tensor of a usual product of Riemannian manifolds if and
only if $D_{M_1}$ and $D_{M_2}$ are totally geodesic foliations.

By virtue of Theorem \ref{d1par}, Theorem \ref{d2par} and Theorem \ref{yataypar}, we have the following theorem;
\begin{theorem}
Let $\pi$ be a semi slant $\xi^\perp-$Riemannian submersion from a Sasakian manifold $(M_1,\varphi,\xi,\eta,g_1)$
onto a Riemannian manifold $(M_2,g_2)$ with a semi-slant angle $\theta.$ Then the total space $M_1$ is a locally product manifold of the leaves of $D_1$, $D_2$ and $(ker\pi_*)^\perp,$ i.e.,
$M_1=M_1{_{D_1}}\times M_1{_{D_2}}\times M_1{_{(ker\pi_*)^\perp}},$ if and only if
\begin{align*}
g_{2}((\nabla \pi_{\ast })(U,\varphi V),\pi _{\ast }\omega Z)&=g_{1}(\mathcal{T}_{U}\omega\phi Z,V),
\end{align*}
\begin{align*}
-g_{2}((\nabla \pi_{\ast })(U,\varphi V),\pi _{\ast }\mathcal{C}X)&=g_{1}(V,\hat{\nabla}_{U}\phi\mathcal{B}X
+\mathcal{T}_{U}\omega\mathcal{B}X)+g_{1}(V,\varphi U)\eta(X),
\end{align*}
\begin{align*}
g_{2}(\pi _{\ast }\omega W,(\nabla \pi_{\ast})(Z,\varphi U))&=g_{1}(\phi W,\hat{\nabla}_{Z}\varphi U),
\end{align*}
\begin{align*}
g_{2}((\nabla \pi _{\ast })(Z,\omega W),\pi _{\ast}(X))-g_{2}((\nabla\pi _{\ast
})(Z,\omega \phi W),\pi _{\ast}(X))&=g_{1}(\mathcal{T}_{Z}\omega W,\mathcal{B}X)+g_{1}(W,\varphi Z)\eta (X)
\end{align*}
and
\begin{align*}
g_{1}(V,\mathcal{V}\nabla^{^{M_1}}_{X}\phi \mathcal{B}Y+\mathcal{A}_{X}\omega\mathcal{B}Y)&=g_{2}(\pi_*(\mathcal{C}Y),(\nabla \pi_*)(X,\varphi V)),
\end{align*}
\begin{align*}
g_{1}(\mathcal{A}_{X}\omega W,\mathcal{B}Y)+\eta(Y)g_{1}(X,\omega W)&=g_{2}((\nabla \pi_*)(X,Y),\pi_*\omega\phi W)
-g_{2}((\nabla \pi_*)(X,\mathcal{C}Y),\pi_*\omega W)
\end{align*}
for $X, Y\in\Gamma((\ker \pi_*)^{\bot}),$ $U, V\in\Gamma(D_1)$ and $Z, W\in\Gamma(D_{2}).$
\end{theorem}

From Theorem \ref{yataypar} and Theorem \ref{dikeypar}, we have the following theorem;
\begin{theorem}
Let $\pi:(M_1,\varphi,\xi,\eta,g_1)\longrightarrow (M_2,g_2)$ be a semi-slant $\xi^\perp-$Riemannian submersion
from a Sasakian manifold $(M_1,\varphi,\xi,\eta,g_1)$ onto a Riemannian manifold $(M_2,g_2)$ with a semi-slant angle $\theta.$ Then the total space $M_1$ is a
locally (usual) product manifold of the leaves of $ker\pi_*$ and $(ker\pi_*)^\perp,$ i.e.,
$M_1=M_1{_{ker\pi_*}}\times M_1{_{(ker\pi_*)^\perp}},$ if and only if
\begin{align*}
g_{1}(V,\mathcal{V}\nabla^{^{M_1}}_{X}\phi \mathcal{B}Y+\mathcal{A}_{X}\omega\mathcal{B}Y)&=g_{2}(\pi_*(\mathcal{C}Y),(\nabla \pi_*)(X,\varphi V)),
\end{align*}
\begin{align*}
g_{1}(\mathcal{A}_{X}\omega W,\mathcal{B}Y)+\eta(Y)g_{1}(X,\omega W)&=g_{2}((\nabla \pi_*)(X,Y),\pi_*\omega\phi W)
-g_{2}((\nabla \pi_*)(X,\mathcal{C}Y),\pi_*\omega W)
\end{align*}
and
\begin{align*}
g_{1}(\omega V,\mathcal{T}_{U}\mathcal{B}X)+g_{1}(V,\phi U)\eta (X)&=g_{2}((\nabla \pi_*)(U,\mathcal{C}X),\pi_*\omega V)-g_{2}((\nabla \pi_*)(U,X),\pi_*\omega\phi V)
\end{align*}
for $X, Y\in \Gamma((\ker \pi_*)^{\bot}), U, V\in\Gamma(D_1)$ and $W\in\Gamma(D_{2}).$
\end{theorem}

\section{Totally umbilical and Totally geodesicness of $\pi$}
In this section, we are going to examine the totally umbilical fibres and the totally geodesicness of a semi-slant $\xi^\perp-$ Riemannian submersion.
First we give a new condition for a semi-slant $\xi^\perp-$Riemannian submersion to be totally umbilical.
Let $\pi$ be a Riemannian submersion from a Riemannian manifold $(M_1,g_1)$ onto a Riemannian manifold $(M_2,g_2)$. $\pi$ is called a Riemannian
submersion with totally umbilical fibres if
\begin{align}
\mathcal{T}_{V}W=g_1(V,W)H,  \label{t1}
\end{align}
for $V, W\in\Gamma(ker\pi_*),$ where $H$ is mean curvature vector field of the
fibres \cite{sahin2}. Then we have the following result.

\begin{theorem}
Let $\pi:(M_1,\varphi,\xi,\eta,g_1)\longrightarrow (M_2,g_2)$ be a semi-slant $\xi^\perp-$Riemannian submersion with totally
umbilical fibres from a Sasakian manifold $(M_1,\varphi,\xi,\eta,g_1)$ onto a Riemannian manifold $(M_2,g_2),$ then $H\in\Gamma(\omega D_2).$
\end{theorem}
\begin{proof}
From (\ref{nvw}), we have
\begin{equation*}
\varphi\nabla^{^{M_1}}_{U_1}U_2=\varphi\mathcal{T}_{U_1}U_2+\varphi\hat{\nabla}_{U_1}U_2
\end{equation*}
for any $U_1,U_2\in\Gamma(D_{1}).$ Now, using (\ref{jv}) and (\ref{jx}), we get
\begin{equation*}
(\nabla^{^{M_1}}_{U_1}\phi )U_2-\nabla^{^{M_1}}_{U_1}\phi U_2=\mathcal{B}\mathcal{T}_{U_1}U_2+\mathcal{C}\mathcal{T}_{U_1}U_2+\phi \hat{\nabla}_{U_1}U_2+\omega \hat{\nabla}_{U_1}U_2.
\end{equation*}
By virtue of (\ref{s2}) and (\ref{s3}) imply that
\begin{equation*}
g_{1}(U_1,U_2)\xi-\eta(U_2)U_1-\mathcal{T}_{U_1}\phi U_2-\hat{\nabla}_{U_1}\phi U_2=\mathcal{B}\mathcal{T}_{U_1}U_2+\mathcal{C}\mathcal{T}_{U_1}U_2+\phi \hat{\nabla}_{U_1}U_2+\omega \hat{\nabla}_{U_1}U_2.
\end{equation*}%
Taking inner product in above equation with $Z\in\Gamma(\mu)$, we obtain
\begin{align*}
g_{1}(U_1,U_2)g_{1}(\xi,Z)=g_1(\mathcal{T}_{U_1}\phi U_2,Z)+g_{1}(\mathcal{C}\mathcal{T}_{U_1}U_2,Z)=g_1(\mathcal{T}_{U_1}\phi U_2,Z)-g_{1}(\mathcal{T}_{U_1}U_2,\varphi Z)
\end{align*}%
From (\ref{t1}), we have
\begin{equation}
g_{1}(U_1,U_2)g_{1}(\xi,Z)=g_{1}(U_1,\phi U_2)g_{1}(H,Z)-g_{1}(U_1,U_2)g_{1}(H,\varphi Z).  \label{t3}
\end{equation}%
Interchanging $U_1$ and $U_2$ in (\ref{t3}), we get
\begin{equation}
g_{1}(U_2,U_1)g_{1}(\xi,Z)=g_{1}(U_2,\phi U_1)g_{1}(H,Z)-g_{1}(U_2,U_1)g_{1}(H,\varphi Z).  \label{t4}
\end{equation}%
Subtracting (\ref{t3}) and (\ref{t4}), we get $g_{1}(H,Z)=0$ which shows that $H\in\Gamma(\omega D_{2}).$
\end{proof}

Now, we give some conditions for a semi-slant $\xi^\perp-$Riemannian submersion from a Sasakian manifold to be totally geodesic map.
Recall that a differential map $\pi$ between two Riemannian manifolds is called totally geodesic if $\nabla \pi_*=0$ \cite{BW}. It is known that the second fundamental form is symmetric.

\begin{theorem}
Let $\pi$ be a semi slant $\xi^\perp-$Riemannian submersion from a Sasakian manifold $(M_1,\varphi,\xi,\eta,g_1)$
onto a Riemannian manifold $(M_2,g_2)$ with a semi-slant angle $\theta.$ Then $\pi$ is a totally geodesic map if
\begin{align}
-\nabla^{\pi}_{X}\pi_{*}Z_2&=\pi_{*}(\mathcal{C}(\mathcal{H}\nabla^{^{M_1}}_{X}\omega Z_{1}-\mathcal{A}_{X}\phi Z_{1}+\mathcal{A}_{X}\mathcal{B}Z_{2}+\mathcal{H}\nabla^{^{M_1}}_{X}\mathcal{C}Z_2)\label{e.q:4.1}\\
&+\omega(\mathcal{A}_{X}\omega Z_{1}-\mathcal{V}\nabla^{^{M_1}}_{X}\phi Z_{1}
+\mathcal{V}\nabla^{^{M_1}}_{X}\mathcal{B}Z_{2}+\mathcal{A}_{X}\mathcal{C}Z_{2})\notag\\
&-\eta(Z_2)\mathcal{C}X-\eta(X)\eta(Z_2)-g_1(Y,\mathcal{C}X)\xi)\notag
\end{align}
for any $X\in\Gamma((ker\pi_*)^\perp)$ and $Z=Z_{1}+Z_{2}\in\Gamma(TM_1),$ where $Z_{1}\in\Gamma(ker\pi_{\ast})$
and $Z_{2}\in\Gamma((ker\pi_{\ast})^{\perp}).$
\end{theorem}
\begin{proof}
By virtue of (\ref{nxv}), (\ref{s2}) and (\ref{s3}), we have
\begin{align*}
\nabla^{^{M_1}}_{X}Z=\varphi(\nabla^{^{M_1}}_{X}\varphi)Z-\varphi\nabla^{^{M_1}}_{X}\varphi Z+\eta(\nabla^{^{M_1}}_{X}Z)\xi
\end{align*}
for any $Z\in\Gamma((ker\pi_*)^\perp)$ and $X\in\Gamma(TM_1)$. Now, from (\ref{npixy}), we arrive
\begin{align*}
(\nabla \pi_{*})(X,Z)&=\nabla^{\pi}_{X}\pi_{*}Z
+\pi_{*}(\varphi\nabla^{^{M_1}}_{X}\varphi Z-\varphi(\nabla^{^{M_1}}_{X}\varphi)Z-\eta(\nabla^{^{M_1}}_{X}Z)\xi)\\
&=\nabla^{\pi}_{X}\pi_{*}Z+\pi_{*}(\varphi(\nabla^{^{M_1}}_{X}\varphi Z_1+\nabla^{^{M_1}}_{X}\varphi Z_2)-\eta(Z)\varphi X-\eta(\nabla^{^{M_1}}_{X}Z)\xi).
\end{align*}
By using (\ref{nxv}), (\ref{nxy}), (\ref{jv}) and (\ref{jx}) we get
\begin{align*}
(\nabla \pi_{*})(X,Z)&=\nabla^{\pi}_{X}\pi_{*}Z_2+\pi_{*}(\mathcal{B}\mathcal{A}_{X}\phi Z_{1}+\mathcal{C}\mathcal{A}_{X}\phi Z_{1}+\phi\mathcal{V}\nabla^{^{M_1}}_{X}\phi Z_{1}
+\omega\mathcal{V}\nabla^{^{M_1}}_{X}\phi Z_{1}\\
&+\phi \mathcal{A}_{X}\omega Z_{1}+\omega \mathcal{A}_{X}\omega Z_{1}+\mathcal{B}\mathcal{H}\nabla^{^{M_1}}_{X}\omega Z_{1}
+\mathcal{C}\mathcal{H}\nabla^{^{M_1}}_{X}\omega Z_{1}\\
&+\mathcal{B}\mathcal{A}_{X}\mathcal{B}Z_{2}+\mathcal{C}\mathcal{A}_{X}\mathcal{B}Z_{2}+\phi\mathcal{V}\nabla^{^{M_1}}_{X}\mathcal{B}Z_{2}
+\omega\mathcal{V}\nabla^{^{M_1}}_{X}\mathcal{B}Z_{2}\\
&+\phi \mathcal{A}_{X}\mathcal{C}Z_{2}+\omega \mathcal{A}_{X}\mathcal{C}Z_{2}+\mathcal{B}\mathcal{H}\nabla^{^{M_1}}_{X}\mathcal{C}Z_{2}
+\mathcal{C}\mathcal{H}\nabla^{^{M_1}}_{X}\mathcal{C}Z_{2}\\
&-\eta(Z_2)\varphi X-\eta(X)\eta(Z_2)-g_1(Z_2,\mathcal{C}X)\xi)
\end{align*}
for any $Z=Z_{1}+Z_{2}\in\Gamma(TM_1)$, where $Z_{1}\in\Gamma(ker\pi_{\ast})$ and $Z_{2}\in\Gamma((ker\pi_{\ast})^{\perp}).$ Thus taking into account the vertical parts, we obtain
\begin{align*}
(\nabla \pi_{*})(X,Z)&=\nabla^{\pi}_{X}\pi_{*}Z_2+\pi_{*}(\mathcal{C}(\mathcal{A}_{X}\phi Z_{1}+\mathcal{H}\nabla^{^{M_1}}_{X}\omega Z_{1}
+\mathcal{A}_{X}\mathcal{B}Z_{2}+\mathcal{H}\nabla^{^{M_1}}_{X}\mathcal{C}Z_{2})\\
&+\omega(\mathcal{V}\nabla^{^{M_1}}_{X}\phi Z_{1}+\mathcal{A}_{X}\omega Z_{1}+\mathcal{V}\nabla^{^{M_1}}_{X}\mathcal{B}Z_{2}+\mathcal{A}_{X}\mathcal{C}Z_{2})\\
&-\eta(Z_2)\mathcal{C}X-\eta(X)\eta(Z_2)-g_1(Z_2,\mathcal{C}X)\xi),
\end{align*}
which gives our assertion.
\end{proof}

\begin{theorem}
Let $\pi$ be a semi slant $\xi^\perp-$Riemannian submersion from a Sasakian manifold $(M_1,\varphi,\xi,\eta,g_1)$
onto a Riemannian manifold $(M_2,g_2)$ with a semi-slant angle $\theta.$ Then $\pi$ is a totally geodesic map if and only if
\begin{enumerate}
\item[(i)]$\begin{aligned}[t]
g_{1}(\hat{\nabla}_{U_{1}}\varphi V_{1},\mathcal{B}Z)=g_{1}(\mathcal{T}_{U_{1}}\mathcal{C}Z,\varphi V_{1})-g_{1}(V_{1},\phi U_{1})\eta(Z),
\end{aligned}$
\item[(ii)] $\begin{aligned}[t]
g_{2}(\nabla\pi _*(U_{2},\omega \phi V_{2}),\pi_*Z)+g_{2}(\nabla \pi_*(U_{2},\omega V_{2}),\pi_*Z)&=g_{1}(\mathcal{T}_{U_{2}}\omega V_{2},\mathcal{B}Z)+g_{1}(V_{2},\phi U_{2})\eta (Z)
\end{aligned}$
\item[(iii)]$\begin{aligned}[t]
g_{2}(\nabla\pi_*(U,\mathcal{C}X),\pi_*\mathcal{C}Y)-g_{2}(\nabla\pi_*(U,\omega\mathcal{B}X),\pi_*Y)&=g_{1}(\mathcal{T}_{U}\phi \mathcal{B}X,Y)-g_{1}(\mathcal{T}_{U}\mathcal{C}X,\mathcal{B}Y) \\
&\!\!+\!\!\eta(X)g_{1}(QU,\varphi Y)\!\!-\!\!\eta(Y)[U\eta(X)\!\!+\!\!g_{1}(X,\omega U)]
\end{aligned}$
\end{enumerate}
for any $U_1, V_1\in\Gamma(D_1),\ U_2,V_2\in\Gamma(D_2),\ U\in\Gamma(ker\pi_*)$ and $X, Y, Z\in\Gamma((ker\pi_{*})^\perp)$.
\end{theorem}
\begin{proof}
(i) Given  $U_1,V_1\in\Gamma(D_1)$ and $Z\in\Gamma((ker\pi_{*})^\perp)$, from (\ref{npixy}), (\ref{s2}) and (\ref{s3}), we have
\begin{align*}
g_{2}((\nabla \pi _{\ast })(U_{1},V_{1}),\pi _{\ast }Z)&=-g_{1}(\varphi\nabla^{^{M_1}}_{U_{1}}V_{1},\varphi Z)-\eta (\nabla^{^{M_1}}_{U_{1}}V_{1})\eta(Z).
\end{align*}
By using (\ref{nxv}) and (\ref{s3}), we obtain
\begin{align*}
g_{2}((\nabla\pi_*)(U_{1},V_{1}),\pi_*Z)&=-g_{1}(\nabla^{^{M_1}}_{U_{1}}\varphi V_{1},\mathcal{B}Z)-g_{1}(\nabla^{^{M_1}}_{U_{1}}\varphi
V_{1},\mathcal{C}Z)-g_{1}(V_{1},\phi U_{1})\eta(Z).
\end{align*}
From (\ref{nvw}), we get
\begin{align*}
g_{2}((\nabla\pi_*)(U_{1},V_{1}),\pi_*Z)&=-g_{1}(\hat{\nabla}_{U_{1}}\varphi
V_{1},\mathcal{B}Z)-g_{1}(\mathcal{T}_{U_{1}}\varphi V_{1},\mathcal{C}Z)-g_{1}(V_{1},\phi U_{1})\eta(Z)
\end{align*}
which gives (i).

(ii) By virtue of (\ref{npixy}), (\ref{s2}) and (\ref{s3}), we get
\begin{align*}
g_{2}((\nabla\pi_*)(U_{2},V_{2}),\pi_*Z)&=-g_{1}(\nabla^{^{M_1}}_{U_{2}}\varphi V_{2},\varphi Z)-\eta (\nabla^{^{M_1}}_{U_{2}}V_{2})\eta(Z)
\end{align*}
for $U_{2},V_{2}\in \Gamma (D_{2}).$ Then using (\ref{jv}), (\ref{jx}) and (\ref{s3}), we obtain
\begin{align*}
g_{2}((\nabla\pi_*)(U_{2},V_{2}),\pi_*Z) &=-g_{1}(\nabla^{^{M_1}}_{U_{2}}\phi V_{2},\varphi Z)-g_{1}(\nabla^{^{M_1}}_{U_{2}}\omega V_{2},\mathcal{B}Z)\\
&-g_{1}(\nabla^{^{M_1}}_{U_{2}}\omega V_{2},\mathcal{C}Z)-g_{1}(V_{2},\phi U_{2})\eta(Z).
\end{align*}
Taking into account that $\pi$ is a semi-slant $\xi^\perp-$Riemannian submersion, using (\ref{nvw}) imply that
\begin{align*}
g_{2}((\nabla\pi_*)(U_{2},V_{2}),\pi_*Z)&=g_{1}(\nabla^{^{M_1}}_{U_{2}}\phi^{2}V_{2},Z)+g_{1}(\pi_*(\mathcal{H}\nabla^{^{M_1}}_{U_{2}}\omega\phi V_{2}),\pi_*Z)-g_{1}(\mathcal{T}_{U_{2}}\omega V_{2},\mathcal{B}Z) \\
&-g_{2}(\pi_*(\mathcal{H}\nabla^{^{M_1}}_{U_{2}}\omega V_{2}),\pi_*(\mathcal{C}Z))-g_{1}(V_{2},\phi U_{2})\eta(Z).
\end{align*}
Now, from (\ref{npixy}) and (\ref{phiv}), we obtain
\begin{align*}
g_{2}((\nabla\pi_*)(U_{2},V_{2}),\pi_*Z)&=\cos^{2}\theta g_{2}((\nabla\pi_*)(U_{2},V_{2}),\pi_*Z)+g_{2}(\nabla\pi_*(U_{2},\omega\phi V_{2}),\pi_*Z)-g_{1}(\mathcal{T}_{U_{2}}\omega V_{2},\mathcal{B}Z) \\
&+g_{2}(\nabla\pi_*(U_{2},\omega V_{2}),\pi_*(\mathcal{C}Z))-g_{1}(V_{2},\phi U_{2})\eta(Z),
\end{align*}%
so with some elementary calculations, we arrive
\begin{align*}
\sin ^{2}\theta g_{2}((\nabla\pi_*)(U_{2},V_{2}),\pi_*Z)&=g_{2}(\nabla\pi_*(U_{2},\omega\phi V_{2}),\pi_*Z)
-g_{1}(\mathcal{T}_{U_{2}}\omega V_{2},\mathcal{B}Z) \\
&+g_{2}(\nabla\pi_*(U_{2},\omega V_{2}),\pi_*(\mathcal{C}Z))-g_{1}(V_{2},\phi U_{2})\eta(Z)
\end{align*}%
which completes (ii).

(iii) If $X,Y\in \Gamma ((\ker \pi _{\ast })^{\bot }),$ $U\in \Gamma (\ker \pi_{\ast })$, then by using (\ref{npixy}), (\ref{s2}) and (\ref{s3}), we have
\begin{align*}
g_{2}((\nabla\pi_*)(U,X),\pi_*Y)&=-g_{1}(\varphi\nabla^{^{M_1}}_{U}X,\varphi Y)-\eta(\nabla^{^{M_1}}_{U}X)\eta(Y).
\end{align*}
Using (\ref{s2}), (\ref{jv}) and (\ref{jx}), we get
\begin{align*}
g_{2}((\nabla\pi_*)(U,X),\pi_*Y)&=-g_{1}(\nabla^{^{M_1}}_{U}\varphi X,\varphi Y)+\eta(X)g_{1}(U,\mathcal{C}Y)-U\eta(X)\eta(Y)-g_{1}(X,\omega U)\eta (Y).
\end{align*}
By virtue of (\ref{nvx}), (\ref{u}), (\ref{jv}) and (\ref{jx}), we obtain
\begin{align*}
g_{2}((\nabla\pi_*)(U,X),\pi_*Y)&=g_{1}(\nabla^{^{M_1}}_{U}\phi \mathcal{B}X,Y)+g_{1}(\nabla^{^{M_1}}_{U}\omega\mathcal{B}X,Y)-g_{1}(\mathcal{T}_{U}\mathcal{C}X,\mathcal{B}Y)
-g_{2}(\pi_*(\mathcal{H}\nabla^{^{M_1}}_{U}\mathcal{C}X),\pi_*\mathcal{C}Y)\\
&+\eta(X)g_{1}(QU,\varphi Y)-U\eta(X)\eta(Y)-g_{1}(X,\omega U)\eta(Y).
\end{align*}%
By (\ref{npixy}) and (\ref{nvw}), we get
\begin{align*}
g_{2}((\nabla \pi _{\ast })(U,X),\pi _{\ast }Y) &=g_{1}(\mathcal{T}_{U}\phi BX,Y)+g_{2}((\nabla\pi_*)(U,\omega\mathcal{B}X),\pi_*Y)
-g_{1}(\mathcal{T}_{U}\mathcal{C}X,\mathcal{B}Y)\\
&+g_{2}((\nabla\pi_*)(U,\mathcal{C}X),\pi_*\mathcal{C}Y)+\eta(X)g_{1}(QU,\varphi Y)\\
&-U\eta(X)\eta(Y)-g_{1}(X,\omega U)\eta(Y).
\end{align*}%
From the above equation, we obtain (iii).
\end{proof}

In a similar way, we obtain the following lemma:

\begin{theorem}
Let $\pi$ be a semi slant $\xi^\perp-$Riemannian submersion from a Sasakian manifold $(M_1,\varphi,\xi,\eta,g_1)$
onto a Riemannian manifold $(M_2,g_2)$ with a semi-slant angle $\theta.$ Then $\pi$ is a totally geodesic map if and only if
\begin{enumerate}
\item [(i)] $\mathcal{C}(\mathcal{T}_{U}\phi V+\nabla^{^{M_1}}_{U}\omega V)+\omega(\hat{\nabla}_{U}\phi V+\mathcal{T}_{U}\omega V)
+g_1(\mathcal{P}V,\phi U)\xi=0.$
\item [(ii)] $\mathcal{C}(\mathcal{A}_{X}\phi U+\mathcal{H}\nabla^{^{M_1}}_{X}\omega U)+\omega(\mathcal{A}_{X}\omega U+\mathcal{V}\nabla^{^{M_1}}_{X}\phi U)+g_1(QU,\mathcal{B}X)\xi=0.$
\item [(iii)] $\mathcal{C}(\mathcal{T}_{U_1}\phi V_1+\mathcal{H}\nabla^{^{M_1}}_{U_1}\phi V_1)+\omega(\mathcal{T}_{U_1}\omega V_1+\mathcal{V}\nabla^{^{M_1}}_{U_1}\phi V_1)=0,$
\end{enumerate}
for $U_1\in\Gamma(D_1),\ V_1\in\Gamma(D_2),\ U,V\in\Gamma(ker\pi_*)$ and $X\in\Gamma((ker\pi_{*})^\perp)$.
\end{theorem}

\section{Examples}

\begin{example}
Every invariant submersion from a Sasakian manifold to
a Riemannian manifold is a semi-slant $\xi^\perp-$Riemannian submersion with $D_2=\{0\}$ and $\theta={0}$.
\end{example}

\begin{example}
Every slant Riemannian submersion from a Sasakian manifold to
a Riemannian manifold is a semi-slant $\xi^\perp-$Riemannian submersion with $D_1=\{0\}$.
\end{example}

Now, we construct some non-trivial examples of semi-slant $\xi^\perp-$Riemannian submersion from a Sasakian manifold.
Let $(\mathbb{R}^{2n+1},g,\varphi ,\xi ,\eta)$ denote the manifold $\mathbb{R}^{2n+1}$ with its usual Sasakian structure given by
\[
\varphi (\sum_{i=1}^{n}(X_{i}\frac{\partial }{\partial x^{i}}+Y_{i}\frac{%
\partial }{\partial y^{i}})+Z\frac{\partial }{\partial z})=%
\sum_{i=1}^{n}(Y_{i}\frac{\partial }{\partial x^{i}}-X_{i}\frac{\partial }{%
\partial y^{i}})
\]%
\[
g=\eta \otimes \eta +\frac{1}{4}\sum_{i=1}^{n}(dx^{i}\otimes
dx^{i}+dy^{i}\otimes dy^{i}),
\]%
\[
\eta =\frac{1}{2}(dz-\sum_{i=1}^{n}y^{i}dx^{i}),\,\,\,\xi =2\frac{\partial }{%
\partial z},
\]%
where $(x^{1},...,x^{n},y^{1},...,y^{n},z)$ are the Cartesian coordinates.
Throughout this section, we will use this notation.
\begin{example}
\label{exm1} Let $F$ be a submersion defined by
\begin{equation*}
\begin{array}{cccc}
F: & \mathbb{R}^{9} & \longrightarrow  & \mathbb{R}^{5} \\
& (x_{1},x_{2},x_{3},x_{4},y_{1},y_{2},y_{3},y_{4},z) &  & (\frac{x_{1}+x_{2}%
}{\sqrt{2}},\frac{y_{1}+y_{2}}{\sqrt{2}},sin\alpha x_{3}-cos\alpha x_{4},y_{4},z)%
\end{array}%
\end{equation*}%
with $\alpha \in (0,\frac{\pi }{2}).$ Then it follows that
\begin{align*}
kerF_{\ast }=span\{& Z_{1}=\frac{\partial}{\partial x^{1}}-\frac{\partial}{%
\partial x^{2}},\ Z_{2}=\frac{\partial}{\partial y^{1}}-\frac{\partial}{%
\partial y^{2}}, \\
& Z_{3}=-\cos\alpha\frac{\partial}{\partial x^{3}}-\sin\alpha\frac{\partial}{\partial x^{4}}%
,Z_{4}=\frac{\partial}{\partial y^{3}}\}
\end{align*}%
and
\begin{align*}
(kerF_{\ast })^{\perp }=span\{& H_{1}=\frac{\partial}{\partial x^{1}}+\frac{%
\partial}{\partial x^{2}},\ H_{2}=\frac{\partial }{\partial y^{1}}+\frac{%
\partial}{\partial y^{2}},\ H_{3}=\sin\alpha\frac{\partial }{\partial x^{3}}-\cos\alpha\frac{%
\partial}{\partial x^{4}}, \\
& H_{4}=\frac{\partial}{\partial y^{4}},\ H_{5}=\frac{\partial}{\partial z}=\xi\}.
\end{align*}%
Hence we have $\varphi Z_{1}=-Z_{2}$, $\varphi Z_{2}=Z_{1}$. Thus it follows that
$D_{1}=span\{Z_{1},Z_{2}\}$ and $D_{2}=span\{Z_{3},Z_{4}\}$ is a slant distribution with slant angle $\theta=\alpha.$
Thus $F$ is a semi-slant submersion with semi-slant angle $\theta.$ Also by direct computations, we obtain
\begin{equation*}
g_{2}(F_*H_{1},F_*H_{1})=g_{1}(H_{1},H_{1}),\ g_{2}(F_*H_{2},F_*H_{2})=g_{1}(H_{2},H_{2}),
\end{equation*}%
\begin{equation*}
g_{2}(F_{\ast }H_{3},F_{\ast }H_{3})=g_{1}(H_{3},H_{3}),\ g_{2}(F_{\ast
}H_{4},F_{\ast }H_{4})=g_{1}(H_{4},H_{4}),\ g_{2}(F_{\ast}\xi ,F_{\ast}\xi)=g_{1}(\xi ,\xi )
\end{equation*}%
where $g_{1}$ and $g_{2}$ denote the standard metrics (inner products) of $%
\mathbb{R}^{9}$ and $\mathbb{R}^{5}$. Thus $F$ is a semi-slant $\xi^{\perp}-$Riemannian submersion.
\end{example}

\begin{example}
Let $F$ be a submersion defined by
\begin{equation*}
\begin{array}{cccc}
F: & \mathbb{R}^{7} & \longrightarrow  & \mathbb{R}^{3} \\
& (x_{1},x_{2},x_{3},y_{1},y_{2},y_{3},z) &  & (\frac{x_2-y_3}{\sqrt{2}},y_{2},z).%
\end{array}%
\end{equation*}%
Then the submersion $F$ is a semi-slant $\xi^\perp-$ Riemannian submersion such that
$D_1=<\frac{\partial}{\partial x_{1}},\frac{\partial}{\partial y_{1}}>$ and
$D_2=<\frac{\partial}{\partial x_{2}}+\frac{\partial}{\partial y_{3}},\frac{\partial}{\partial x_{3}}>$
with semi-slant angle $\alpha =\frac{\pi }{4}.$
\end{example}

\begin{example}
Let $F$ be a submersion defined by
\begin{equation*}
\begin{array}{cccc}
F: & \mathbb{R}^{9} & \longrightarrow  & \mathbb{R}^{3} \\
& (x_{1},x_{2},x_{3},x_{4},y_{1},y_{2},y_{3},y_{4},z) &  & (sin\alpha x_{3}-cos\alpha x_{4},y_{4},z)%
\end{array}%
\end{equation*}%
with $\alpha \in (0,\frac{\pi }{2}).$ Then the submersion $F$ is a semi-slant $\xi^\perp-$ Riemannian submersion such that
$D_1=<\frac{\partial}{\partial x_{1}},\frac{\partial}{\partial x_{2}},%
\frac{\partial}{\partial y_{1}},\frac{\partial}{\partial y_{2}}>$ and
$D_2=<-\cos\alpha\frac{\partial}{\partial x_{3}}-\sin\alpha
\frac{\partial}{\partial x_{4}},\frac{\partial}{\partial y_{3}}>$
with semi-slant angle $\theta=\alpha.$
\end{example}

\begin{example}
Let $F$ be a submersion defined by
\begin{equation*}
\begin{array}{cccc}
F: & \mathbb{R}^{13} & \longrightarrow  & \mathbb{R}^{7} \\
& (x_{1},x_{2},x_{3},x_{4},x_{5},x_{6},y_{1},y_{2},y_{3},y_{4},y_{5},y_{6},z) &  & (\frac{x_1-x_2}{\sqrt{2}},\frac{y_1-y_2}{\sqrt{2}},\frac{x_3+x_4}{\sqrt{2}},\frac{y_3+y_4}{\sqrt{2}},
\frac{x_5-x_6}{\sqrt{2}},y_{5},z).%
\end{array}%
\end{equation*}%
Then the submersion $F$ is a semi-slant $\xi^\perp-$ Riemannian submersion such that
$D_1=<\frac{\partial}{\partial x_{1}}+\frac{\partial}{\partial x_{2}},%
\frac{\partial}{\partial y_{1}}+\frac{\partial}{\partial y_{2}},
\frac{\partial}{\partial x_{3}}-\frac{\partial}{\partial x_{2}},
\frac{\partial}{\partial y_{3}}-\frac{\partial}{\partial y_{4}}>$ and
$D_2=<\frac{\partial}{\partial x_{5}}+\frac{\partial}{\partial x_{6}}, \frac{\partial}{\partial y_{6}}>$
with semi-slant angle $\alpha =\frac{\pi }{4}.$\\
\end{example}

\noindent{\bf{Acknowledgement}}\newline
This paper is supported by Bing\"{o}l University research project (BAP-FEF.2016.00.011).

\end{document}